\documentclass[12pt,a4paper]{amsart}
\usepackage{mathtools}
\usepackage{xcolor}
\usepackage{amsmath,amstext,amssymb,amsthm,amsfonts}
\usepackage[all]{xy} \xyoption{poly}

\newcommand{\nc}{\newcommand}

\nc{\one}{\mbox{\bf 1}}
\nc{\invtensor}{\underset{\leftarrow}{\otimes}}
\nc{\const}{\operatorname{const}}

\nc{\ad}{\operatorname{ad}}
\nc{\tr}{\operatorname{tr}}

\nc{\tp}{\operatorname{top}}
\nc{\rank}{\operatorname{rank}}
\nc{\corank}{\operatorname{corank}}
\nc{\codim}{\operatorname{codim}}
\nc{\sdim}{\operatorname{sdim}}
\nc{\mult}{\operatorname{mult}}

\nc{\spn}{\operatorname{span}}
\nc{\Sym}{\operatorname{Sym}}
\nc{\sym}{\operatorname{sym}}
\nc{\id}{\operatorname{id}}
\nc{\Id}{\operatorname{Id}}
\nc{\Ree}{\operatorname{Re}}

\nc{\htt}{\operatorname{ht}}

\nc{\str}{\operatorname{str}}

\nc{\Ker}{\operatorname{Ker}}
\nc{\rker}{\operatorname{rKer}}
\nc{\im}{\operatorname{Im}}
\nc{\osp}{\mathfrak{osp}}
\nc{\sgn}{\operatorname{sgn}}
\nc{\F}{\operatorname{F}}
\nc{\Mod}{\operatorname{Mod}}
\nc{\DS}{\operatorname{DS}}
\nc{\Soc}{\operatorname{Soc}}
\nc{\Inj}{\operatorname{Inj}}
\nc{\Hom}{\operatorname{Hom}}
\nc{\End}{\operatorname{End}}
\nc{\supp}{\operatorname{supp}}
\nc{\Card}{\operatorname{Card}}
\nc{\Ann}{\operatorname{Ann}}
\nc{\Ind}{\operatorname{Ind}}
\nc{\Coind}{\operatorname{Coind}}
\nc{\wt}{\operatorname{wt}}
\nc{\ch}{\operatorname{ch}}
\nc{\sch}{\operatorname{sch}}

\nc{\Stab}{\operatorname{Stab}}
\nc{\Sch}{{\mathcal S}\mbox{\em ch}}
\nc{\Irr}{\operatorname{Irr}}
\nc{\Spec}{\operatorname{Spec}}
\nc{\Prim}{\operatorname{Prim}}
\nc{\Aut}{\operatorname{Aut}}
\nc{\Ext}{\operatorname{Ext}}

\nc{\Fract}{\operatorname{Fract}}
\nc{\gr}{\operatorname{gr}}
\nc{\deff}{\operatorname{def}}
\nc{\HC}{\operatorname{HC}}

\nc{\red}{\operatorname{red}}

\nc{\wdchi}{\widetilde{\chi}}
\nc{\wdH}{\widetilde{H}}
\nc{\wdN}{\widetilde{N}}
\nc{\wdM}{\widetilde{M}}
\nc{\wdO}{\widetilde{O}}
\nc{\wdR}{\widetilde{R}}

\nc{\wdV}{\widetilde{V}}

\nc{\wdC}{\widetilde{C}}

\nc{\Obj}{\operatorname{Obj}}
\nc{\Dglie}{\operatorname{{\mathcal D}glie}}
\nc{\Fin}{\operatorname{{\mathcal F}in}}

\nc{\Adm}{\operatorname{\mathcal{A}dm}}

\nc{\Sg}{{\cS(\fg)}}
\nc{\Shg}{{\cS(\fhg)}}
\nc{\Ug}{{\cU(\fg)}}
\nc{\Uhg}{{\cU(\fhg)}}
\nc{\Sh}{{\cS(\fh)}}
\nc{\Uh}{{\cU(\fh)}}
\nc{\Uhh}{{\cU(\fhh)}}
\nc{\Zg}{{{\mathcal{Z}}(\fg)}}

\nc{\Vir}{{\mathcal{V}ir}}
\nc{\NS}{{\mathcal{N}S}}

\nc{\tZg}{{\widetilde{\mathcal Z}({\mathfrak g})}}
\nc{\Zk}{{\mathcal Z}({\mathfrak k})}

\nc{\Up}{{\mathcal U}({\mathfrak p})}
\nc{\Ah}{{\mathcal A}({\mathfrak h})}
\nc{\Ag}{{\mathcal A}({\mathfrak g})}
\nc{\Ap}{{\mathcal A}({\mathfrak p})}
\nc{\Zp}{{\mathcal Z}({\mathfrak p})}
\nc{\cR}{\mathcal R}
\nc{\cS}{\mathcal S}
\nc{\cT}{\mathcal{T}}
\nc{\cX}{\mathcal X}
\nc{\cA}{\mathcal A}
\nc{\cU}{\mathcal U}
\nc{\cZ}{\mathcal Z}
\nc{\cM}{\mathcal M}
\nc{\cL}{\mathcal L}
\nc{\cF}{\mathcal F}
\nc{\fg}{\mathfrak g}

\nc{\fo}{\mathfrak o}

\nc{\CO}{\mathcal O}
\nc{\CR}{\mathcal R}
\nc{\Cl}{\mathcal {C}\ell}

\nc{\cW}{\mathcal{W}}
\nc{\bM}{\mathbf{M}}
\nc{\bL}{\mathbf{L}}
\nc{\bN}{\mathbf{N}}

\nc{\zq}{\mathpzc q}

\nc{\fl}{\mathfrak l}
\nc{\fn}{\mathfrak n}
\nc{\fm}{\mathfrak m}
\nc{\fp}{\mathfrak p}
\nc{\fh}{\mathfrak h}
\nc{\ft}{\mathfrak t}
\nc{\fk}{\mathfrak k}
\nc{\fb}{\mathfrak b}
\nc{\fs}{\mathfrak s}
\nc{\fB}{\mathfrak B}

\nc{\vareps}{\varepsilon}
\nc{\varesp}{\varepsilon}
\nc{\veps}{\varepsilon}

\nc{\fsl}{\mathfrak{sl}}
\nc{\fgl}{\mathfrak{gl}}
\nc{\fso}{\mathfrak{so}}
\nc{\fpq}{\mathfrak{pq}}
\nc{\fq}{\mathfrak q}
\nc{\fsq}{\mathfrak{sq}}
\nc{\fpsl}{\mathfrak{psl}}

\nc{\fhg}{\hat{\fg}}
\nc{\fhn}{\hat{\fn}}
\nc{\fhh}{\hat{\fh}}
\nc{\fhb}{\hat{\fb}}
\nc{\hrho}{\hat{\rho}}

\nc{\hsl}{\hat{\fsl}}
\nc{\fpo}{\mathfrak{po}}
\nc{\dirlim}{\underset{\rightarrow}{\lim}\,}
\nc{\nen}{\newenvironment}
\nc{\ol}{\overline}
\nc{\ul}{\underline}
\nc{\ra}{\rightarrow}
\nc{\lra}{\longrightarrow}
\nc{\Lra}{\Longrightarrow}
\nc{\bo}{\bar{1}}
\nc{\Lla}{\Longleftarrow}

\nc{\Llra}{\Longleftrightarrow}

\nc{\thla}{\twoheadleftarrow}

\nc{\lang}{(}
\nc{\rang}{)}

\nc{\hra}{\hookrightarrow}

\nc{\iso}{\overset{\sim}{\lra}}

\nc{\ssubset}{\underset{\not=}{\subset}}

\nc{\vac}{|0\rangle}

\nc{\Thm}[1]{Theorem~\ref{#1}}
\nc{\Prop}[1]{Proposition~\ref{#1}}
\nc{\Lem}[1]{Lemma~\ref{#1}}
\nc{\Cor}[1]{Corollary~\ref{#1}}
\nc{\Conj}[1]{Conjecture~\ref{#1}}
\nc{\Claim}[1]{Claim~\ref{#1}}
\nc{\Defn}[1]{Definition~\ref{#1}}
\nc{\Exa}[1]{Example~\ref{#1}}
\nc{\Rem}[1]{Remark~\ref{#1}}
\nc{\Note}[1]{Note~\ref{#1}}
\nc{\Quest}[1]{Question~\ref{#1}}
\nc{\Hyp}[1]{Hypoth\`ese~\ref{#1}}
\nen{thm}[1]{\label{#1}{\bf Theorem.\ } \em}{}
\nen{prop}[1]{\label{#1}{\bf Proposition.\ } \em}{}
\nen{lem}[1]{\label{#1}{\bf Lemma.\ } \em}{}
\nen{cor}[1]{\label{#1}{\bf Corollary.\ } \em}{}
\nen{conj}[1]{\label{#1}{\bf Conjecture.\ } \em}{}

\nen{claim}[1]{\label{#1}{\bf Claim.\ } \em}{}

\nen{defn}[1]{\label{#1}{\bf Definition.\ } }{}
\nen{exa}[1]{\label{#1}{\bf Example.\ } }{}
\nen{rem}[1]{\label{#1}{\em Remark.\ } }{}
\nen{note}[1]{\label{#1}{\em Note.\ } }{}
\nen{exer}[1]{\label{#1}{\em Exercise.\ } }{}
\nen{sket}[1]{\label{#1}{\em Sketch of proof.\ } }{}
\nen{quest}[1]{\label{#1}{\bf Question.\ } \em}{}

\nen{hyp}[1]{\label{#1}{\bf Hypoth\`ese.\ } \em}{}
\begin{document}
\setcounter{section}{-1}
\setcounter{tocdepth}{1}

\title{On DS functor for affine Lie superalgebras}
\author{Maria Gorelik,  Vera Serganova }

\address[]{Dept. of Mathematics, The Weizmann Institute of Science,Rehovot 7610001, Israel}
\email{maria.gorelik@weizmann.ac.il}

\address[]{Dept. of Mathematics,
University of California at Berkeley, Berkeley CA 94720}
\email{serganov@math.berkeley.edu}
\thanks{Supported in part by BSF Grant 2012227.}

\begin{abstract}
We study Duflo-Serganova functor for non-twisted affine Lie superalgebras
and affine vertex superalgebras.
\end{abstract}
\maketitle

\section{Introduction}
Let $\fg$ be a Lie superalgebra and let $x\in\fg_{\ol{1}}$ satisfy the condition 
$[x,x]=0$. The operator $\ad_x$ defines an odd square zero endomorphism of any 
$\fg$-module. This yields a functor $N\mapsto \DS_x(N):=Ker_N x/Im_N x$ from the category of 
$\fg$-modules to the category of modules over $\fg_x:=\DS_x(\fg)$.

The functor $\DS_x$ was introduced in \cite{DS} (see also~\cite{S2}) as a means to 
assign an analog of singular support to representations of Lie superalgebras.
This functor preserves superdimension and tensor product of representations.

Recall that the {\em defect} of a finite-dimensional Lie superalgebra $\fg$
is the dimension
of a maximal isotropic subspace in $\mathbb{Q}\Delta$; for $A(m-1,n-1), B(m,n), D(m,n)$ the defect is
equal to $min (m,n)$; for other cases of non Lie algebras it is one.
It is well-known that
the defect is equal to the maximal number of mutually orthogonal isotropic simple roots.
A finite-dimensional simple Lie superalgebra of zero defect
is either a simple Lie algebra or $\mathfrak{osp}(1|2l)$; the finite-dimensional modules
over these Lie superalgebras are completely reducible (and these are the only simple
Lie superalgebras with this property).

If $\fg$ is a  finite-dimensional  Lie superalgebra, then $\DS_x(\fg)$
is a  finite-dimensional  Lie superalgebra of a smaller defect.
If $\fg$ is the affinization of $\dot{\fg}$ and $x\in \dot{\fg}_{\ol{1}}$, then
$\fg_x$ is the affinization of $\dot{\fg}_x$, see~\cite{GS}.

In this paper we consider the DS functors  for  affine Lie superalgebras $\fg=\dot{\fg}^{(1)}$
and affine vertex superalgebras $V_k(\fg)$; we always assume that
$x\in \dot{\fg}_{\ol{1}}$.

Let $Vac^k(\fg)$ be a vacuum $\fg$-module of level $k$ and
$Vac_k(\fg)$ be its simple quotient.
It is easy to see that $\DS_x(Vac^k(\fg))=Vac^k(\DS_x(\fg))$.
We prove that for a non-negative integral $k$ one has
$\DS_x(Vac_k(\fg))=Vac_k(\DS_x(\fg))$ if $\dot{\fg}_x$
has zero defect and $\dot{\fg}_x\not=\mathbb{C}$, see~\Thm{thmABC}.
As a result, the corresponding vertex algebras are isomorphic, see~\Cor{corvert}.

The principal admissible modules for an affine Lie algebra $\ft$   were classified in~\cite{KW5}. A level $k$ is called
principal admissible if  $Vac_k(\ft)$ is principal admissible.
From Theorem of Arakawa~\cite{A} it follows that
for a principal admissible level $k$
the $V_k(\ft)$-modules in the category $\CO$ are completely reducible and the irreducible
 modules are  the principal admissible modules of level $k$.

We introduce the principal admissible levels for
an affine Lie superalgebra $\fg$ using Kac-Wakimoto definition for Lie algebra case.
We prove that   if $\dot{\fg}_x$ is a simple Lie algebra and $\dot{\fg}\not=B(n+1|n)$,
then for a principal admissible level $k$ one has $\DS_x(Vac_k(\fg))=Vac_k(\fg_x)$.
This implies the isomorphism of the corresponding vertex algebras.
The proof is based on Arakawa's Theorem and the fact that the maximal proper submodule
in $Vac^k(\fg_x)$ is generated by a singular vector (if $\fg_x$ is a Lie algebra,
this can be easily deduced from~\cite{F}).
We believe that the statement holds for $\fg_x=\mathfrak{osp}(1|2n)^{(1)}$, however
both Arakawa's and Fiebig's results are not established in this case.

In Section~\ref{sect1} we recall the construction of Duflo-Serganova functor
$\DS_x$ and summarize the results which we use later.

In Section~\ref{sect2} we study DS functor for integrable vacuum modules and
prove~\Thm{thmABC}.  Since integrable vacuum modules have principal admissible levels,
this theorem for the case, when $\dot{\fg}_x$ is a simple Lie algebra and $\dot{\fg}\not=B(n+1|n)$,
is a particular case of~\Thm{thmadmlevel}. However, the proof of~\Thm{thmABC} is different:
it does not use vertex algebras and Arakawa's Theorem.
In~\S~\ref{n+n-} we give an example when $\DS_x(Vac_k(\fg))\not=Vac_k(\fg_x)$
($\fg=\fsl(1|2)^{(1)}$, $k$ is critical).

In Section~\ref{sect3} we introduce the DS functor for vertex superalgebras.
In particular, we prove that if $\DS_x(Vac_k(\fg))=Vac_k(\fg_x)$, then
$\DS_x$ maps  the simple affine vertex superalgebra $V_k(\fg)$
to the simple affine  vertex superalgebra $V_k(\fg_x)$.
As a result, for any $V_k(\fg)$-module $N$ the image $\DS_x(N)$ is a  $V_k(\fg_x)$-module.

In Section~\ref{sect4} we study $Vac_k(\fg)$ if $k$ is a principal admissible level,
(this notion we define in~\S~\ref{admsup} similarly to the Lie algebra case).
In~\S~\ref{veramdsup} we prove that
$\DS_x(Vac_k(\fg))=Vac_k(\fg_x)$
if $\dot{\fg}_x$  is a simple Lie algebra and $\dot{\fg}\not=B(n+1|n)$.

Let $\dot{\Sigma}$ be a set of simple roots which contains
a maximal isotropic subset $S=\{\beta_1,\ldots,\beta_r\}$ ($r$ is the defect of
$\dot{\fg}$).
We consider $\DS_x$ for
$x=\sum_{i=1}^r x_i$, where
$x_i$ is a non-zero vector in $\fg_{\beta_i}$.
In this case $\dot{\fg}_x=\DS_x(\dot{\fg})$  has zero defect.
All our results are valid also for the composition $DS_S:=DS_{x_1}\circ DS_{x_2}\circ \ldots \circ DS_{x_r}$. Note that $DS_S(\fg)\cong \DS_x(\fg)$.

Partial results of this paper were reported at the conferences
in ESI in January 2017 and in RIMS in October 2017.

{\em Acknowledgment.}
We are grateful to T.~Arakawa, V.~Kac and M.~Wakimoto for helpful discussions.
A part of this work was done during  the first author stay in ESI in January 2017
and in MIT in August 2017.
We are grateful to these institutions for stimulating atmosphere and excellent
working conditions.

\section{Preliminaries}\label{sect1}
Throughout the paper  $\fg=\dot{\fg}^{(1)}$, where
$\dot{\fg}\not=D(2|1,a)$ is  a finite-dimensional Kac-Moody
superalgebra with a set of simple roots (a base) $\dot{\Sigma}$ and
a Cartan subalgebra $\dot{\fh}$. We denote by $\fh$ the Cartan subalgebra of
$\fg$: $\fh=\dot{\fh}\oplus\mathbb{C}K\oplus\mathbb{C}d$.

Let $\Delta\subset\fh^*$ (rep., $\dot{\Delta})$ be the set of roots of $\fg$ (resp., of $\dot{\fg}$). We denote by $\Delta_{\ol{0}}$ and $\Delta_{\ol{1}}$ the
subsets of even and odd roots. We denote by $W$ the Weyl group of $\fg_{\ol{0}}$.
Recall that $\Delta_{\ol{0}}$ is a  union of
a finite number of root systems  of affine Lie algebras
with the same minimal imaginary root $\delta$.
Throughout the paper we fix $\dot{\Delta}$ and
denote by $\Lambda_0$ the corresponding fundamental weight, i.e.
$(\Lambda_0,\delta)=1$ and $(\Lambda_0,\dot{\Delta})=(\Lambda_0,\Lambda_0)=0$.

We assume that $\dot{\Delta}$ is indecomposable
(i.e., $\fg, \dot{\fg}$ are quasisimple in the sense of~\cite{S3}).
We fix $\Delta_{\ol{0}}^+$
and consider the subsets of positive roots $\Delta^+$
which contain $\Delta_{\ol{0}}^+$. The choice of $\Delta^+$ gives a triangular decomposition
of $\fg$, compatible with the triangular decomposition of $\fg_{\ol{0}}$,
corresponding to $\Delta_{\ol{0}}^+$. For a fixed subset of positive roots $\Delta^+$
we denote by $\Sigma$ the corresponding base (i.e., the set of simple roots)
and by $\rho$  the corresponding Weyl vector.
For a fixed base $\Sigma$  we denote by $\alpha_0$ the affine root, i.e.
$\Sigma=\dot{\Sigma}\cup\{\alpha_0\}$, where $\dot{\Sigma}$ is the base of $\dot{\Delta}^+$.

We denote by $\CO$ the BGG category of finitely generated $\fg$-modules
with a diagonal action of $\fh$ and a locally finite action of $\fg_{\alpha}$
with $\alpha\in\Delta_{\ol{0}}^+$. The category $\CO$ is equipped by a
duality functor $N\mapsto N^{\sharp}$ and the simple modules are self-dual.

We normalize the form $(-,-)$ on $\fg$ as in~\cite{KW3}
and set
$$\dot{\Delta}^{\#}:=\{\alpha\in\dot{\Delta}_{\ol{0}}| \ (\alpha,\alpha)>0\}.$$
The corresponding algebra $\dot{\fg}^{\#}$ is a simple Lie algebra;
we denote its highest root by $\theta$.
We will use bases $\dot{\Sigma}$ such that  $\theta$ is the highest
root in ${\Delta}^+(\dot{\Sigma})$; then
$$\alpha_0=\delta-\theta.$$

Let $\dot{\Omega}$ be the Casimir operator for $\dot{\fg}$ which corresponds to the invariant bilinear form $(-,-)$, see~\cite{Kbook2}, Ch. II.
Recall that the dual Coxeter number $h^{\vee}$
is half of the eigenvalue of the Casimir operator $\dot{\Omega}$ on the adjoint representation $\dot{\fg}$ and
$h^{\vee}=(\rho,\delta)$. We always choose the Weyl vector $\rho$ in the form $\rho=h^{\vee}\Lambda_0+\dot{\rho}$.

We say that $k\in\mathbb{C}$ is non-critical if $k\not=-h^{\vee}$
and $\lambda\in\fh^*$ is non-critical if $K(\lambda)\not=-h^{\vee}$, i.e.
$(\lambda+\rho,\delta)\not=0$.

We use the following notations: if $X,Y\subset\fh^*$ we set $(X,Y)=\{(x,y)|\ x\in X,y\in Y\}$; for a vector space $V$,
 $X\subset V$ and $R\subset\mathbb{C}$ we use the notation
$RX=\{\sum_{i=1}^s r_ix_i|\ r_i\in R, x_i\in X\}$ (for instance, $\mathbb{Z}\Delta$
is the root lattice). For $S\subset \fh^*$ we set $$S^{\perp}:=\{\nu\in\fh^*|\ \forall\beta\in S\ (\beta,\nu)=0\}.$$

If $\alpha\in\Delta$ is a non-isotropic root,
we say that a $\fg$-module $N$ is $\alpha$-integrable if $\fg_{\pm\alpha}$ act locally nilpotently on $N$. We use conventions of~\cite{GK}.
We say that a $\fg$-module $N\in\CO$ is {\em integrable} if
$N$ is integrable as a $\dot{\fg}_{\ol{0}}$-module and
$N$ is $\alpha$-integrable
for each $\alpha\in\Delta$ satisfying $||\alpha||^2>0$.

 For $A(m|1)^{(1)}, C(m)^{(1)}$,  $N\in\CO$ is integrable if $N$ is integrable as
a $\fg_{\ol{0}}$-module and $\fh$ acts diagonally.

\subsection{DS functor for affine Lie superalgebras}
Take $x\in\fg_{\ol{1}}$ satisfying $[x,x]=0$.
Recall that Duflo-Serganova functor $\DS_x$ is defined by
$$\DS_x(N):=Ker_N x/Im_N x;$$
we view $\DS_x(N)$ as a module over $\fg^x$
(where $\fg^x$ is the centralizer of $x$ in $\fg$).
Note that $[x,\fg]\subset \fg^x$ acts trivially on $\DS_x(N)$ and
that $\fg_x:=\DS_x(\fg)=\fg^x/[x,\fg]$ is a Lie superalgebra. Thus
$\DS_x(N)$ is a $\fg_x$-module and $\DS_x$ is a functor from the category of $\fg$-modules to the category of $\fg_x$-modules. This is a tensor functor
($\DS_x(N\otimes N')=\DS_x(N)\otimes \DS_x(N')$, see~\cite{DS}).

An exact sequence of $\fg$-modules
$$0\to N_1\to N\to N_2\to 0$$
induces the exact sequence of $\fg_x$-modules
$$0\to E\to \DS_x(N_1)\to \DS_x(N)\to \DS_x(N_2)\to \Pi(E)\to 0.$$

Recall that for a $\fg$-module $N$ with a diagonal action of $\fh$
one has
$$\sch N:=\sum_{\nu\in\fh^*} \sdim N_{\nu}e^{\nu}.$$
If $0\to N_1\to N\to N_2\to 0$ is exact, then
$\sch \DS_x(N)=sch  \DS_x(N_1)+sch \DS_x(N_2)$.

\subsection{Choice of $x$}\label{choices}
In this paper we consider $\DS_x$ for $x\in\dot{\fg}$:
\begin{equation}\label{eqx}
x\in\dot{\fg}_{\ol{1}}, \ \ [x,x]=0.
\end{equation}

\subsubsection{Definition}
For $a\in {\fg}$ we denote by $\supp(a)$ the subset of ${\Delta}\cup\{0\}$
such that
$$a=\sum_{\beta\in\supp(a)} a_{\beta},$$
where $a_{\beta}$ is a non-zero vector in $\fg_{\beta}$.

\subsubsection{Definition}
We call $S\subset \dot{\Delta}_{\ol{1}}$ an {\em isotropic set} if $S$ is a basis
of an isotropic  subspace in $\fh^*$.

Note that if $\supp(x)$ is an isotropic set, then $x$ satisfies~(\ref{eqx}).

\subsubsection{}
Let $\dot{G}$ be the Lie group
 of $\dot{\fg}_{\ol{0}}$. By~\cite{DS}, Thm. 4.2
each $x$ satisfying~(\ref{eqx}) is
$\dot{G}$-conjugate to $x'$, where $\supp(x')$ is an isotropic set;
this gives a  one-to-one correspondence
between the $\dot{G}$-orbits for $x$ satisfying~(\ref{eqx}) and $\dot{W}$-orbit of isotropic sets in $\dot{\Delta}$. In particular, for
each $x$ satisfying~(\ref{eqx}) there exists a base $\dot{\Sigma}$ such that
$x$ is $\dot{G}$-conjugate to $x'$ such that $\supp(x')$ is an isotropic set and
$\supp(x')\subset\dot{\Sigma}$.
We call the cardinality of $\supp(x')$ the {\em rank} of $x$; $0$ has the zero rank and
the maximal rank is equal to the defect of $\dot{\fg}$.

Let $\ft$ be a Lie subalgebra of $\dot{\fg}_{\ol{0}}$ and
$N$ be a $\fg$-module which is $\ft$-finite (i.e., $\cU(\ft)v$ is finite-dimensional for each $v\in N$). Then the Lie group of $\ft$ acts on $N$. Moreover,
any element $g$ in this Lie group  induces
an isomorphism between the algebras $\fg_x$ and $\fg_{Ad_g(x)}$
and the corresponding modules $\DS_x(N)$ and $\DS_{Ad_g(x)}(N)$.

In particular, for a $\dot{\fg}$-integrable $\fg$-module $N$, this construction
gives an isomorphism between $\DS_x(N)$ and $\DS_{x'}(N)$ with $x'$ as above
($\supp(x')$ is an isotropic subset of a certain base $\dot{\Sigma}$).

\subsubsection{}\label{Sigmax}
Assume that $S:=\supp(x)$ is an isotropic set.

It is shown in~\cite{DS}, Lemma 6.3  that $\dot{\fg}_x$
 a finite-dimensional Kac-Moody superalgebra with the roots
$$\dot{\Delta}_x:=(S^{\perp}\cap\dot{\Delta})\setminus(S\cup(-S));$$
$\dot{\fg}_x$ can be identified
with a subalgebra of $\dot{\fg}$ generated by the root spaces
$\fg_{\alpha}$ with $\alpha\in\dot{\Delta}_x$ and
$\dot{\fh}_x\subset\dot{\fh}^x=\{h\in\dot{\fh}|\ S(h)=0\}$ such that
$$\dot{\fh}_x\oplus (\sum_{\beta\in
  S}\mathbb{C}h_{\beta})=\dot{\fh}^x,\,\,[\fg_{\alpha},\fg_{\alpha}]\subset
\dot{\fh}_x\,\,\forall \alpha\in \dot{\Delta}_x.$$
Moreover, $\dot{\fh}_x$
is a Cartan subalgebra of $\dot{\fg}_x$ and
$\dot{\fg}^x=\dot{\fg}_x\oplus [x,\dot{\fg}]$.
Note that $\fh^*_x$
is identified with a subspace in $\fh^*$ and  $S^{\perp}=\mathbb{C}S\oplus \fh^*_x$.

If $\dot{\Delta}_x$ is not empty, then
$\dot{\Delta}_x$ is the root system of the Lie superalgebra $\dot{\fg}_x$.
One can choose a set of simple roots $\dot{\Sigma}_x$ such that $\Delta^+(\dot{\Sigma}_x)=\Delta^+\cap \dot{\Delta}_x$.

Let $r$ be the rank of $x$ (i.e., $|S|=r$).
If $\dot{\fg}=A(m|n), B(m|n)$ or $D(m|n)$, then $\dot{\fg}_x=A(m-r|n-r), B(m-r|n-r)$ or $D(m-r|n-r)$ respectively.
If $\dot{\fg}=C(n)$, $G_3$ or $F_4$, then $r=1$ and $\dot{\fg}_x$ is the Lie algebra of type $C_{n-2}$, $A_1$ and $A_2$ respectively.
If  $\dot{\fg}=D(2|1;a)$, then $r=1$ and $\dot{\fg}_x=\mathbb C$.
One has
$$\ \text{defect}\ \dot{\fg}_x=\text{defect}\ \dot{\fg}-\text{rank} x.$$

It is easy to show (see~\cite{GS}) that $\DS_x(\fg)=\fg_x$ is the affinization of $\dot{\fg}_x$; we identify this algebra with
$$\fg_x=\sum_{s=-\infty}^{\infty} (\dot{\fg}_x t^s)\oplus\mathbb{C}K\oplus\mathbb{C}d,
\ \ \fh_x:=\dot{\fh}\oplus\mathbb{C}K\oplus\mathbb{C}d;$$
then $\Delta_x:=\Delta(\fg_x)$ is the affinization of $\dot{\Delta}_x$. One has
$$\fh^*_x=\dot{\fh}^*\oplus\mathbb{C}\delta\oplus\mathbb{C}\Lambda_0\subset\fh^*,
\ \ S^{\perp}=\fh^*_x\oplus \mathbb{C}S.$$

Set $\Delta^+_x:=\Delta^+(\Sigma)\cap \Delta_x$
and consider the corresponding triangular decomposition of $\fg_x$.
We will describe the base $\Sigma_x$ which corresponds to $\Delta^+_x$
below.

If $\dot{\Delta}_x$ is empty, then $\dot{\fg}_x=0$ or $\dot{\fg}_x=\fgl_1$.
If $\dot{\fg}_x=0$ (i.e., $\dot{\fg}=A(n|n), A(n+1|n),
B(n|n), D(n|n), C(2)$), then $\fg_x=\mathbb{C}K\times\mathbb{C}d$.
If $\dot{\fg}_x=\fgl_1$ (i.e., $\dot{\fg}=D(n+1|n)$ or $D(2|1,a)$), then
  $\fg_x=\fgl_1^{(1)}$, $\Delta^+(\fg_x)=\mathbb{Z}_{>0}\delta$ and $\Sigma_x=\{\delta\}$.

 If $\dot{\Sigma}_x$ is connected, then
$\Sigma_x:=\dot{\Sigma}_x\cup\{\delta-\theta_x\}$,
where $\theta_x$ is the maximal root in $\Delta^+(\dot{\Sigma}_x)$.

If $\dot{\Sigma}_x$ is not connected, then
$\dot{\Delta}_x=D_2$ (i.e., $\dot{\fg}=D(n+2|n)$ with $x$ of the maximal rank).
In this case $\Delta_x=D_2^{(1)}$ is a union of two copies of
$A_1^{(1)}$ with the same imaginary roots, that is
$\Sigma_x:=\dot{\Sigma}_x\cup\{\delta-\theta_x^i\}_{i=1}^2$, where
$\dot{\Delta}^+_x=\{\theta_x^1,\theta_x^2\}$.

\subsection{Casimir operator}\label{Casimir}
Take $x$ as in~(\ref{eqx}).
The bilinear form $(-,-)$ induces an invariant bilinear form $(-,-)_x$
on $\fg_x$.
If $N$ is an integrable $\fg$-module, then $\DS_x(N)$ is an integrable
$\fg_x$-module.

Let $\Omega$ be the Casimir operator for $\fg$ which corresponds to the invariant bilinear form $(-,-)$, see~\cite{Kbook2}, Ch. II.
Let $\dot{\Delta}_x\not=\emptyset$. By~\cite{GS},
the image of $\Omega$ is the Casimir operator for $\fg_x$.
This implies  $||\rho||^2=||\rho_x||^2_x$ and
\begin{equation}\label{eqCasimir}
[\DS_x(L_{\fg}(\lambda)):L_{\fg_x}(\lambda')]\not=0\ \Longrightarrow\
\ (\lambda+2\rho,\lambda)=(\lambda'+2\rho_x,\lambda')_x,
\end{equation}
where $\lambda'\in \fh_x^*=S^{\perp}/\mathbb{C}S$.

\subsection{Duality}\label{dual}
The duality in $\CO$ is defined by an anti-automorphism $\sigma$ of $\fg$
which stabilizes the elements of $\fh$. By above, $\fg_x,\fg_{\sigma(x)}$ are identified with a subalgebra of $\fg$ which is $\sigma$-stable
(in particular, $\fg_x=\fg_{\sigma(x)}$). It si not hard to see that the map
 $\Psi: \DS_x(N^{\sharp})\to (DS_{\sigma(x)}(N))^{\sharp}$ defined by
$\Psi(f)(v):=f(v)$ is an isomorphism of
$\fg_x$-modules if $N\in\CO$.

\section{Integrable vacuum modules}\label{sect2}
In this section $\dot{\fg}$ is a finite-dimensional Kac-Moody algebra
and $x\in\dot{\fg}_{\ol{1}}$ is such that $supp(x)$ has a maximal rank, that is
 $\dot{\fg}_x$ has zero defect.
Recall that $\fg_x$ is the affinization of $\dot{\fg}_x$.

\subsection{Vacuum modules}\label{vacmod}
If $\fp$ is a Kac-Moody superalgebra with a Cartan subalgebra $\ft$, we denote by $L_{\fp}(\lambda)$ a simple highest $\fp$-module with the highest weight $\lambda\in\ft^*$.
For an affine Kac-Moody superalgebra $\fp$ we denote by $Vac^k_{\fp}$ the vacuum module of level $k$ and by $\vac$ the vacuum vector.
 For $\fp=\fg$ we write simply $L(\lambda)$, $Vac^k$.
Note that $L(k\Lambda_0)=Vac_k$ is the simple quotient of
$Vac^k$ and so it does not depend
on the choice of $\Sigma$; we call $L(k\Lambda_0)$ a {\em simple vacuum module};
if $L(k\Lambda_0)$ is integrable, we call it an {\em integrable vacuum module}.

\subsubsection{}
The character of an integrable vacuum module is given by the
Kac-Wakimoto character formula, see~\cite{GK}. From the proof
it follows that an integrable vacuum module
is a unique integrable quotient of $Vac^k$ (since the proof
uses only the fact that $L(k\Lambda_0)$ is a $\fg^{\#}$-integrable quotient of
$Vac^k$, so all integrable quoteints have the same character and thus
such quotient is unique).

Recall that $\dot{\fg}^{\#}\not=D_2$.
Let $\theta$ be the highest root of $\dot{\Delta}^{\#}$
and $e\in\dot{\fg}^{\#}$ be the corresponding root vector.
 From~\cite{Kbook2}, Lem. 3.4 it follows that
 a quotient $Vac^k/I$ is $\fg^{\#}$-integrable if and only if $k\in\mathbb{Z}_{\geq 0}$
and $I$ contains $f_0^{k+1}\vac$
for $f_0:=et^{-1}$. Therefore  $L(k\Lambda_0)$ is integrable if and only if $k\in\mathbb{Z}_{\geq 0}$; in this case
$$L(k\Lambda_0)=Vac^k/I(k),\ \text{
where $I(k)$ is generated by $f_0^{k+1}\vac$}.$$
Note that the vector $f_0^{k+1}\vac$ is singular if $\delta-\theta\in\Sigma$.

\subsubsection{Remark}\label{choicexnew}
Recall that $L(k\Lambda_0)$ does not depend on the choice of $\dot{\Sigma}$.
Combining~\S~\ref{choices} and~\S~\ref{app1}, we see that
computing $\DS_x(L(k\Lambda_0))$ we can always assume that
$supp(x)$ is an isotropic set which lies in $S$ satisfying (P1), (P2), (P3) in ~\S~\ref{app1}.

\subsection{}
\begin{thm}{thmABC}
Let  $x\in\dot{\fg}_{\ol{1}}$ be such that $[x,x]=0$ and $supp(x)$ has a maximal rank.

(i)  If $\dot{\fg}_x=0$ and $k\not=-h^{\vee}$, then
$\DS_x(L(k\Lambda_0))$ is one-dimensional.

(ii) Assume that $\dot{\Sigma}$
contains $S:=\supp(x)$ and the following inclusion holds
\begin{equation}\label{QpiS}
(\mathbb{Q}_{\geq 0}\Sigma\cap S^{\perp})\subset (\mathbb{Q}S+\mathbb{Q}_{\geq 0}\Sigma_x).
\end{equation}
If $L(\lambda)$ is integrable and $(\lambda,S)=0$, then
$$\DS_x(L(\lambda))\cong L_{\fg_x}(\lambda|_{\fh_x}).$$

(iii) If $\dot{\fg}_x\not=\mathbb{C}$ and $k\in\mathbb{Z}_{\geq 0}$, one has
$$ \DS_x(L(k\Lambda_0))\cong L_{\fg_x}(k\Lambda_0).$$
\end{thm}

\subsection{Proof of~\Thm{thmABC}}
For (i), (iii) we set $\lambda:=k\Lambda_0\in\fh^*$ (so $L(\lambda)=L(k\Lambda_0)$) and
$S:=supp(x)$. Using~\Rem{choicexnew}, we assume for (i), (iii) that
$S,\dot{\Sigma}$ satisfies (P1), (P3) of~\S~\ref{app1}, i.e.
$S\subset\dot{\Sigma}$ and~(\ref{QpiS}) holds except for
$\dot{\fg}=D(n+2|n), D(n+1|n)$.

We introduce
$$\lambda':=\lambda|_{\fh_x}\in\fh_x^*.$$

Since $(\lambda,S)=0$ one has
$\dim L(\lambda)_{\lambda-\nu}=\delta_{0,\nu}$ for $\nu\in\mathbb{Z}S$.
Thus the singular vector in $L(\lambda)$ has a non-trivial image in $\DS_x(L(k\lambda))$
which is singular; moreover,
\begin{equation}\label{DS1}
[\DS_x(L(\lambda)):L_{\fg_x}(\lambda')]=1.
\end{equation}

For (i)  $\fg_x=\mathbb{C}K\times\mathbb{C}d$.
By~\S~\ref{Casimir}, the Casimir  $\Omega_x=2(K+h^{\vee})d$ acts on
$\DS_x(L(k\Lambda_0))$ by a scalar, so $d$ acts on $\DS_x(L(k\Lambda_0))$ by a scalar.
Now (i) follows from~(\ref{DS1}).

For (ii), (iii) assume that
$[\DS_x(L(\lambda)):L_{\fg_x}(\lambda'-\nu')]\not=0$
for some $\nu'\in \fh^*_x$ with $\nu'\not=0$. Since $\lambda'-\nu'$
is a weight of $\DS_x(L(\lambda))$,
there exists $\nu\in\fh^*$ such that $\nu|_{\fh_x}=\nu'$, $(\nu,S)=0$
and $L(\lambda)_{\lambda-\nu}\not=0$. In particular,
\begin{equation}\label{SigmaS}
\nu\in \mathbb{Z}_{\geq 0}\Sigma\cap S^{\perp}.
\end{equation}

Let us prove (ii). Combining~(\ref{SigmaS}) and~(\ref{QpiS}), $\nu\in\mathbb{Q}S+\mathbb{Q}_{\geq 0}\Sigma_x$, that is
$\nu'\in \mathbb{Q}_{\geq 0}\Sigma_x$.
If $\Delta_x$ is empty, we obtain $\nu'=0$, a contradiction.
Now we assume that $\Delta_x$ is not empty. By~\S\ref{Casimir}, $\DS_x(L(\lambda))$
is $\fg_x$-integrable, so $L_{\fg_x}(\lambda'),L_{\fg_x}(\lambda'-\nu')$
are integrable modules and $||\lambda'-\nu'+\rho_x||^2=||\lambda'+\rho_x||^2$,
that is
$$(\lambda'-\nu'+\rho_x,\nu')+(\lambda'+\rho_x,\nu')=0.$$
Since $\fg_x$ has zero defect, the integrability of
 $L_{\fg_x}(\lambda')$ and $L_{\fg_x}(\lambda'-\nu')$
gives $(\lambda',\nu'),  (\lambda'-\nu',\nu')\geq 0$ and $(\nu',\rho_x)>0$
(for $\nu'\not=0$), a contradiction. This establishes (ii).
Recall that for our choice of $(S,\dot{\Sigma})$, (\ref{QpiS}) holds except for
$\dot{\fg}=D(n+2|n), D(n+1|n)$. Thus (ii) implies (iii).

It remains to verify (iii) for $\dot{\fg}=D(n+2|n)$.
Consider the short exact sequence
$$0\to I(k)\to Vac^k\to L(k\Lambda_0)\to 0,$$
where $I(k)$ is the maximal proper submodule of $Vac^k$.
It is easy to see that $\DS_x(Vac^k)=Vac^k(\fg_x)$. Thus
the corresponding long exact sequence is
$$0\to E\to \DS_x(I(k))\xrightarrow{\phi}Vac^k(\fg_x) \xrightarrow{\psi}  \DS_x(L(k\Lambda_0))\to \Pi(E)\to 0.$$
Since $\DS_x(L(k\Lambda_0))$ is $\fg_x$-integrable, the image of
$\psi$ is an integrable quotient of $Vac^k$, that is $L_{\fg_x}(k\Lambda_0)$.
Since $L_{\fg_x}(k\Lambda_0-\nu')$ is a subquotient of $\DS_x(L(k\Lambda_0))$
and $\nu'\not=0$, it is a subquotient of $\Pi(E)$. Therefore
$\Pi(L_{\fg_x}(k\Lambda_0-\nu'))$ is a subquotient of
$\DS_x(I(k))$.  Take $\Sigma$ such that $||\alpha_0||^2>0$.
By~\S~\ref{vacmod}, $I(k)$ is generated by a singular vector of
the weight $k\Lambda_0-(k+1)\alpha_0$. Therefore
$$\lambda'-\nu'=k\Lambda_0-(k+1)\alpha_0-\mu',$$
where $\mu'=\mu|_{\fg_x}$ for some $\mu\in\fh^*$ such that
$\mu\in\mathbb{Z}_{\geq 0}\Sigma$. Then
\begin{equation}\label{jk}
(\nu',\Lambda_0)\geq k+1.
\end{equation}

Take $S:=\{\vareps_{i+1}-\delta_i\}_{i=1}^n$ and
$\dot{\Sigma}=\{\vareps_1-\vareps_2,\vareps_2-\delta_1,\ldots,
\vareps_{n+1}-\delta_n,\delta_n\pm\vareps_{n+2}\}$; then
 $\alpha_0=\delta-\vareps_1-\vareps_2$,
so $||\alpha_0||^2=2$. One has
$$\begin{array}{ll}
\Sigma_x=\{\vareps_1\pm \vareps_{n+2};\delta-(\vareps_1\pm \vareps_{n+2})\},\ \ \rho_x=2\Lambda_0+\vareps_1.
\end{array}$$

One readily sees that $(\mathbb{C}\Sigma\cap S^{\perp})\subset(\mathbb{C}S+\mathbb{C}\Sigma_x)$,
so $\nu'\in \mathbb{C}\Sigma_x$, so
$$\nu'=j\delta-s_+(\vareps_1+\vareps_{n+2})-s_-(\vareps_1-\vareps_{n+2}).$$
The integrability of $L_{\fg_x}(k\Lambda_0-\nu')$ implies
$0\leq s_{\pm}\leq k/2$. In addition,~(\ref{eqCasimir}) gives
$$(k+2)j-s_+-s_-=s_+^2+s_-^2,$$
so   $j\leq k/2$. However, $j=(\nu',\Lambda_0)\geq k+1$ by~(\ref{jk}). This contradiction completes the proof.

\subsection{Example: $\fg=\fsl(1|2)^{(1)}, k=-1$}\label{n+n-}
Take $\fg=\fsl(1|2)^{(1)}$ with $\Sigma=\{\delta-\vareps_1+\delta_2, \vareps_1-\delta_1,\delta_1-\delta_2\}$ and $S=\{\vareps_1-\delta_1\}$.
Using the character formula (3.20) in~\cite{KW4}
 it is not hard to show that $\DS_x((L(-\Lambda_0))$
is not one-dimensional.

\section{DS functor for vertex superalgebras}\label{sect3}
\subsection{Vertex algebras}
Recall that a vertex (super)algebra $V=V_{\ol{0}}\oplus V_{\ol{1}}$ is a  vector superspace endowed with a vacuum vector
$\vac$,  an even linear endomorphism $T$ and a parity preserving  linear map

$$Y:V\to (\End V)[[z,z^{-1}]],\ \ a\mapsto Y(a,z)=\sum_{n\in\mathbb{Z}} a_{(n)}z^{-n-1}$$

subject to the following axioms ($a,b\in V, m,n\in\mathbb{Z}$)

(translation covariance) $[T, Y(a,z)]=\partial_z Y(a,z)$;

(vacuum) $T\vac=0;\ \ Y(\vac,z)=Id_V;\ \ a_{(-1)}\vac=a,\ a_{(n)}\vac=0\ \text{ for }n\geq 0;$

and the locality axiom which we use in the Borcherds form

$$(a_{(m)}b)_{(n)}=\sum_{i=0}^{\infty} (-1)^i \binom{m}{i}
\bigl(a_{(m-i)}b_{(n+i)}-(-1)^{m+p(a)p(b)}b_{(m+n-i)}a_{(i)}\bigr).$$

For $m=0$ this gives
\begin{equation}\label{commab}
(a_{(0)}b)_{(n)}=[a_{(0)},b_{(n)}].
\end{equation}

Note that $T$ is ``determined'' by $Y$, i.e.
\begin{equation}\label{TTT}
Ta=a_{(-2)}\vac.
\end{equation}

\subsubsection{Modules}\label{vertmod}
A weak module over a vertex superalgebra $V$ in a vector superspace $M$ with a parity preserving linear map
$$Y^M:V\to (End M)[[z,z^{-1}]]\ \
\ \ a\mapsto Y_M(a,z)=\sum_{n\mathbb{Z}} a_{(n)}^M z^{-n-1},$$
such that for each $v\in M$ one has $ a_{(n)}^M v=0$ for $n>>0$,
$Y_M(\vac,z)=Id_M$ and $a_{(m)}^M, b^M_{(n)}$ satisfy
the locality axiom. As above, the locality axiom gives
\begin{equation}\label{commabM}
(a_{(0)}b)_{(n)}^M=[a^M_{(0)},b^M_{(n)}].
\end{equation}

An ideal of a vertex algebra is a  subspace $I\subset V$ such that
$a_{(m)}b, b_{(m)}a\in I$ for each $a\in I, b\in V,  m\in\mathbb{Z}$.
If $I$ is an ideal of $V$, then the quotient $V/I$ inherits the structure of a vertex algebra.

If $I$ is an ideal of $V$, the $V/I$-modules are the $V$-modules
 annihilated by $I$, that is $a_{(m)}^M=0$ for each $a\in I, m\in\mathbb{Z}$.
We say that an ideal $I$ is generated by a set $E$ if $I$ is a minimal ideal
containing $E$. The locality axiom implies that
 in this case $M$ is $V/I$-module
if and only if $a_{(m)}^M=0$  for each $a\in E, m\in\mathbb{Z}$.

\subsection{Definition of $\DS_x(V)$}
Let $V$ be a vertex superalgebra and
 $x\in V$ be such that
\begin{equation}\label{xvert}
x\in V_{\ol{1}},\ \  x_{(0)}x=0\ \text{ and }\ \vac\not\in Im x_{(0)}.
\end{equation}
By~(\ref{commab}) one has $x_{(0)}^2=0$.
We define the vector space $\DS_x(V)$ as follows:
$$\DS_x(V):=\Ker_{V} x_{(0)}/ Im_V x_{(0)}.$$

By the vacuum axiom, $\vac\in\Ker x_{(0)}$, so
$\vac$ has a non-zero image $\vac'\in \DS_x(V)$.
From the translation axiom  $[T,x_{(0)}]=0$, so
$T$ induces an even map $T'\in End \DS_x(V)$.

Take $b\in \Ker_V x_{(0)}$.
From~(\ref{commab}) it follows that $[b_{(n)}, x_{(0)}]=0$ for each $n$, so
$b_{(n)}$ induces $b'_{(n)}\in End (\DS_x(V))$ and $b'_{(n)}=0$
if $b\in Im_V x_{(0)}$.
This gives a parity preserving linear map
$$\DS_x(V) \to  (End \DS_x(V))[[z,z^{-1}]]\ \ \ b\mapsto Y'(b,z)=\sum_{n\in\mathbb{Z}} b'_{(n)}z^{-n-1}.$$

The space $\DS_x(V)$ equipped with $\vac', T'$ and the fields $Y'(b,z)$ form
a vertex algebra  (the
axioms for $V'$ follow
from the corresponding axioms for the vertex algebra $V$).
We denote this vertex algebra by $\DS_x(V)$.

\subsubsection{Remark}
A vertex algebra $V$  is $\mathbb{Z}_{\geq 0}$-graded if $V=\oplus_{s=0}^{\infty} V_s$ with $$deg(a_{(j)}b)=deg(a)+deg(b)-j-1,$$
where $deg$ stands for the degree of a homogeneous vector in $V$.
We claim that
the condition $\vac\not\in Im a_{(0)}$ holds for each $a\in V$
if $V$ is a $\mathbb{Z}_{\geq 0}$-graded vertex algebra with $V_0=\mathbb{C}\vac$.

Indeed, assume that $deg(a_{(0)}b)=0$ for homogeneous $a,b$. Then
 $deg(a)+deg(b)=1$, so $a$ or $b$ lie in $V_0=\mathbb{C}\vac$.
However, $\vac_{(0)}=0$ and $a_{(0)}\vac=0$ for each $a$, that is
$a_{(0)}b=0$. Hence $\vac\not\in Im a_{(0)}$ for each $a\in V$.

\subsubsection{Modules}
Let $M$ be a $V$-module.
The condition $x_{(0)}x=0$ gives $[x_{(0)}^M,x_{(0)}^M]=0$.
We introduce
$$\DS_x(M)=\Ker_M x_{(0)}^M/ Im_M x_{(0)}^M.$$

Using~(\ref{commabM}) it is easy to check
that $\DS_x(M)$ inherits a structure of $\DS_x(V)$-module
(i.e., $Y^M$ induces a map $\DS_x(V)\to (End \DS_x(M))[[z,z^{-1}]]$
which satisfy the corresponding axioms).

\subsection{Affine vertex superalgebras}
Let $\dot{\fg}$ be a finite-dimensional Kac-Moody superalgebra and
let $\fg=\dot{\fg}^{(1)}$.

By~\cite{FZ}, the
vacuum module $Vac^k(\fg)$ has a structure of a vertex superalgebra with
\begin{equation}\label{Yxt}
Y(at^{-1}\vac,z)=\sum_{n\in \mathbb{Z}} (at^n)z^{-n-1}\ \text{ for }a\in\dot{\fg}.
\end{equation}
We denote this vertex superalgebra by $V^k(\fg)$.

The weak $V^k(\fg)$-modules are the restricted $[\fg,\fg]$-module of level $k$
($M$ is "restricted" if for each $v\in M$ one has $(\dot{\fg} t^s)v=0$ for $s>>0$).

The above correspondence between  $V^k(\fg)$-modules and $[\fg,\fg]$-modules
implies that the maximal proper submodule $I(k)$ of  $Vac^k(\fg)$
is the maximal ideal in the vertex algebra $V^k(\fg)$.
Moreover, if $I(k)$ is generated by $E$ as a $\fg$-module, then
$I(k)$ is generated by $E$ as an ideal in $V^k(\fg)$. In particular,
$L(k\Lambda_0)$ inherits a structure of
a vertex superalgebra, which is simple; it is denoted by
$V_k(\fg)$.

For $\dot{\fg}=0$ or $\dot{\fg}=\mathbb{C}$, the vacuum module
$Vac^k(\fg)$ is  one-dimensional and $V^k(\fg)=V_k(\fg)$
is a one-dimensional vertex algebra.

If $\fg$ is a Lie algebra, then
for $k\in\mathbb{Z}_{\geq 0}$ the $V_k(\fg)$-modules
correspond to the restricted integrable $\fg$-modules of level $k$,
see~\cite{DLM}, Thm. 3.7; these
modules are completely reducible and the irreducible modules are
the integrable highest weight modules of level $k$ (there are infinitely many
such modules and $V_k(\fg)$ is a {\em rational} vertex algebra).
The following result was proven in~\cite{GS}, Thm. 6.3.1 for $\dot{\fg}^{\#}\not=D_2$.

\subsubsection{}
\begin{thm}{thmVkmod}
If $L(k\Lambda_0)$ are integrable, then $V_k(\fg)$-modules
are the restricted $[\fg,\fg]$-module of level $k$ which are
$[\fg^{\#}, \fg^{\#}]$-integrable.
\end{thm}
\begin{proof}
Let $I(k)$ be the maximal proper submodule of $Vac^k$, so  $L(k\Lambda_0)=Vac^k/I(k)$.

For $\dot{\fg}^{\#}\not=D_2$ consider
 the natural embedding
$Vac^k(\fg^{\#})\subset Vac^k$ and denote by  $I^{\#}$
the maximal proper submodule of $Vac^k(\fg^{\#})$.

If $\dot{\fg}^{\#}=D_2=A_1\times A_1$ consider  the natural embeddings
$Vac^k(A_1)', Vac^k(A_1)''\subset Vac^k$ which correspond to two copies of $A_1$
in $D_2$; let $I',I''$ be the maximal proper submodules in $Vac^k(A_1)', Vac^k(A_1)''$
respectively. Set $I^{\#}:=I'+I''$.

By~\S~\ref{vacmod}, the submodule
$I(k)$ is generated by $I^{\#}$. By above, a
restricted $[\fg,\fg]$-module $N$ of level $k$
is $V_k(\fg)$-module if and only if it is annihilated by $a_{(m)}$
for each $a\in I^{\#}$, $m\in\mathbb{Z}$. Since $\dot{\fg}^{\#}$ is a Lie algebra,
$N$ is annihilated by $a_{(m)}$
for each $a\in I^{\#}$, $m\in\mathbb{Z}$ if and only if $N$ is $\fg^{\#}$-integrable
(\cite{DLM}, Thm. 3.7).
\end{proof}

\subsection{$\DS_x$ for affine vertex algebras}
Fix $x\in\dot{\fg}_{\ol{1}}$ satisfying $[x,x]=0$. View
$$x':=xt^{-1}\vac$$
as a vector in $V^k(\fg)$ and $V_k(\fg)$ respectively.
Note that $x'_{(0)}=x$.

One has $x'_{(0)}x'=x(xt^{-1}\vac)=0$.
The vertex algebras $V^k(\fg), V_k(\fg)$ are $\mathbb{Z}_{\geq 0}$-graded
(the grading is given by the action of $-d\in\fg$) and the zero component is spanned
by $\vac$. Hence $x'$ satisfies~(\ref{xvert}).

Consider the vertex algebras $\DS_{x'}(V^k(\fg)),
\DS_{x'}(V_k(\fg))$.

It is easy to see that $\DS_x(Vac^k(\fg))$ is canonically isomorphic to
$Vac^k(\fg_x)$ as a $\fg_x$-module.
Choose a vacuum vector $\vac$ in $Vac^k(\fg)$
and let the vacuum vector $\vac'$ in $Vac^k(\fg_x)$ be the image of $\vac$.

\subsubsection{}
\begin{thm}{thmDSvert}
Let $\fg$ be an affine (non-twisted) Lie superalgebra and
let $x\in\dot{\fg}_{\ol{1}}$ be such that $[x,x]=0$; set $x':=xt^{-1}\vac$.

(i) The canonical isomorphism $\DS_x(Vac^k(\fg))\iso Vac^k(\fg_x)$ induces
a vertex algebra isomorphism
 $DS_{x'}(V^k(\fg))\iso V^k(\fg_x)$.

(ii) If $\DS_x(L_{\fg}(k\Lambda_0))\cong L_{\fg_x}(k\Lambda_0)$, then
$\iota$ induces the vertex algebra isomorphism
$$DS_{x'}(V_k(\fg))\iso V_k(\fg_x).$$
\end{thm}
\begin{proof}
By above, $\iota$ is an isomorphism of $\fg_x$-modules and
$\iota(\vac)=\vac'$. If $V^k(\fg_x)$ is one-dimensional, this implies (i) and (ii). Assume that $V^k(\fg_x)$ is not one-dimensional.
Then $\Delta(\fg_x)\not=\emptyset$. Since $\iota$ is an isomorphism of $\fg_x$-modules,
$$\iota(at^{-1}\vac)=(\DS_x(a)t^{-1})\vac'$$
for each $a\in\dot{\fg}$ such that $[x,a]=0$.
By~(\ref{Yxt}) we obtain
\begin{equation}\label{Yiota}
Y(\iota(v),z)=\iota(Y(v,z))
\end{equation}
for each $v=bt^{-1}\vac$  with $b\in\dot{\fg}_x$.

Let $V$ be a vertex algebra and $E$ be a subspace of $V$.
Denote by $\langle E\rangle$ the smallest subspace $V'$ of $V$ which contains $E$
and such that $b_{(j)}v\in \langle E\rangle$ for each $b,v\in V'$ and
$j\in\mathbb{Z}$.
The locality axiom and~(\ref{TTT}) imply that if $V$ admits
two vertex algebra structures $(\vac, T, Y)$ and $(\vac, T', Y')$
such that
$Y(v,z)=Y'(v,z)$ for each $v\in E$, then these structures coincide
on $\langle E\rangle$ (i.e., $TV=T'v$ and $Y(v,z)=Y'(v,z)$
for each $v\in\langle E\rangle$).

Now let $E\subset V^k(\fg_x)$ (resp., $E\subset V^k(\fg_x)$)
be the span of $\vac$ and $bt^{-1}\vac$  with $b\in\dot{\fg}_x$.
Since $V^k(\fg_x)$ and $V_k(\fg_x)$ are generated by $\vac$ as a
$[\fg_x,\fg_x]$-modules, $V^k(\fg_x)=\langle E\rangle$
(resp., $V^k(\fg_x)=\langle E\rangle$).
Thus  $\iota$ is an isomorphism of the vertex algebras.
\end{proof}

Using Theorem~\ref{thmABC} we obtain the
\subsubsection{}
\begin{cor}{corvert}
If $k$ is a non-negative integer,
$x$ has a maximal rank and $\dot{\fg}$
differ from $D(n+1|n), D(2|1,a)$, then
$DS_{x'}(V_k(\fg))$ and $V_k(\fg_x)$ are isomorphic as vertex algebras.
\end{cor}

\subsubsection{}
Take $x$ as above. Let $M$ be a
weak  $V^k(\fg_x)$-module which we view as a restricted
$\fg_x$-modules of level $k$. One readily sees that, as a $\fg_x$-module,
the $DS_{x'}(V^k(\fg))$-module
$DS_{x'}(M)$   is  $\DS_x(M)$, so $DS_{x'}$ for $V^k(\fg)$-modules
correspond to $\DS_x$ for $[\fg,\fg]$-modules.
We will denote the functor $DS_{x'}$ by $\DS_x$.

\subsubsection{}
\begin{cor}{}
Let $L(k\Lambda_0)$ be integrable. For any $V_k(\fg)$-module $M$
the $V^k(\fg_x)$-module
$\DS_x(M)$ is a $V_k(\fg_x)$-module.
\end{cor}
\begin{proof}
By~\Thm{thmVkmod}, $M$ is $[\fg^{\#},\fg^{\#}]$-integrable.
Note that
$\fg^{\#}_x$ is the image of $\fg^{\#}\cap Ker_{\fg} x$ in $\fg_x=Ker_x\fg/ Im_x\fg$.
Therefore
$\DS_x(M)$ is $[\fg^{\#}_x,\fg^{\#}_x]$-integrable, so $\DS_x(M)$
is  a $V_k(\fg_x)$-module by~\Thm{thmVkmod}.
\end{proof}

\section{Principal admissible vacuum modules}\label{sect4}
In this section we define admissible weights for affine Lie superalgebras
and prove~\Thm{thmadmlevel}.

\subsection{Affine Lie algebra case}\label{affineLie}
Let $\dot{\ft}$ be a finite-dimensional simple Lie algebra; let
$\ft=\dot{\ft}^{(1)}$ be  the corresponding
affine Lie algebra with a Cartan subalgebra $\fh$.
We denote by $\Delta_{re}$  the set of real roots of $\ft$.

\subsubsection{}
For a non-critical weight $\lambda\in\fh^*$ the set of
$\lambda$-integral real roots is defined as
$$\Delta_{re}(\lambda)=\{\alpha\in\Delta_{re}|\ \
\frac{2(\lambda+\rho,\alpha)}{(\alpha,\alpha)}\in\mathbb{Z}\}.$$
 For our purposes we consider only $\lambda$s where
$\mathbb{C}\Delta_{re}(\lambda)=\mathbb{C}\Delta_{re}$. In this case
$\Delta_{re}(\lambda)$
is the set of real roots of an affine Lie algebra  algebra $\ol{\ft}$
with the same Cartan algebra $\fh$ and the triangular decomposition
induced by the triangular decomposition of $\ft$, i.e.
$$\Delta_{re}(\lambda)^+:=\Delta_{re}(\lambda)\cap\Delta^+.$$
We denote by $\rho,\ol{\rho}$ the Weyl vectors of $\ft,\ol{\ft}$ respectively.
 The character of $L_{\ft}(\lambda)$ and the character of the
highest weight $\ol{\ft}$-module
$L_{\ol{\ft}}(\lambda+\rho-\ol{\rho})$ are related by the following formula:
\begin{equation}\label{01}
Re^{\rho}\ch L_{\ft}(\lambda)=\ol{R} e^{\ol{\rho}}\ch L_{\ol{\ft}}(\lambda+\rho-\ol{\rho}),
\end{equation}
where $R,\ol{R}$ stand for the respective  Weyl denominators
(see~\cite{KT1},\cite{KT2} and references there).

\subsubsection{Admissible weights}
 A non-critical weight
$\lambda\in\fh^*$ is called {\em admissible } if
$\mathbb{C} \Delta_{re}(\lambda)=\mathbb{C}\Delta_{re}$ and
$L_{\ol{\ft}}(\lambda+\rho-\ol{\rho})$ is an integrable $\ol{\ft}$-module.

If $\lambda$ is admissible, then
$ch L_{\ol{\ft}}(\lambda+\rho-\ol{\rho})$ is given by the Weyl-Kac character formula and
$\ch L(\lambda)$, suitably normalized, is a ratio of theta functions,
which is a modular function, see~\cite{KW1},\cite{KW2}.
The admissible weights were classified in~\cite{KW2}.
An admissible weight $\lambda$ (and a module $L_{\ft}(\lambda)$) is  called {\em principal admissible} if
$\Delta_{re}(\lambda)\cong\Delta_{re}$, that is $\ol{\ft}\cong \ft$;
the principal admissible weights were classified in~\cite{KW5}.

\subsubsection{Principal admissible levels}
A level $k$ is called {\em principal admissible } if $k\Lambda_0$
is principal admissible.

It is easy to see  that $k$ is principal admissible if and only if
$$k+h^{\vee}=\frac{p+h^{\vee}}{u}, \text{ where } p,u\in\mathbb{Z}_{\geq 0}, u>0\ \
(p+h^{\vee},u)=(u,r^{\vee})=1,$$
where $r^{\vee}$ is the lacity of $\dot{\ft}$ (see the definition below in \ref{lacity}).

\subsubsection{}\label{ThmAr}
The following Adamovi\'c-Milas conjecture~\cite{AM}  was proven by T.~Arakawa in~\cite{A}.

{\em Theorem, Arakawa, 2014.}

Let $k$ be an admissible level for an affine Lie algebra $\ft$.
The $V_k(\ft)$-modules in the category $\CO$ are completely reducible and the irreducible
 modules are $L_{\ft}(\lambda)$, where $\lambda$ are the principal admissible weights of level $k$.

\subsection{Admissibility for affine Lie superalgebras}\label{admsup}
Let $\fg$ be a (non-twisted) affine Lie superalgebra.

\subsubsection{Lacity}\label{lacity}
Let $\dot{\fg}\not=D(2|1,a)$. We call
$\alpha\in\Delta$ (resp., $\alpha\in\dot{\Delta}$) a short root if
$|(\alpha,\alpha)|$ takes the smallest non-zero value.
We define the  lacity for $\fg$ and for $\dot{\fg}$ as
$$r^{\vee}=\frac{2}{|(\alpha,\alpha)|},$$
where $\alpha$ is a short root.
Observe that the lacities for  $\Delta$ and for $\dot{\Delta}$ are equal.
Moreover, this lacity is equal to
the lacity of $\dot{\fg}^{\#}$ if $\dot{\fg}\not=B(0|n)$;
 for $\dot{\fg}=B(0|n)$ one has
$r^{\vee}=4$.

The set ${\Delta}_{re}(\lambda)$ was introduced in~\cite{GK}.
As for Lie algebra case, we define the admissible weights as follows.

\subsubsection{Definitions}

A non-critical weight
$\lambda\in\fh^*$ is {\em admissible } if
$\mathbb{C} \Delta_{re}(\lambda)=\mathbb{C}\Delta_{re}$ and
$L_{\ol{\fg}}(\lambda+\rho-\ol{\rho})$ is an integrable $\ol{\fg}$-module.

An admissible weight $\lambda$ is called {\em principal admissible }
if $\Delta_{re}(\lambda)\cong \Delta_{re}$.

We say that $k$ is an {\em admissible }(resp., {\em  principal admissible })
level  if $k\Lambda_0$ is admissible (resp.,  principal admissible).

By~\cite{GK}, Thm. 11.2.3, $\ch L(k\Lambda_0)$ is given by~(\ref{01})
if $k$ is admissible.

\subsection{Principal admissible levels}\label{admsuper}
It is not hard to show that for $\dot{\fg}\not=D(2|1,a)$
the level $k$ is principal admissible if and only if
$$k+h^{\vee}=\frac{p+h^{\vee}}{u}, \text{ where } p,u\in\mathbb{Z}_{\geq 0}, u>0\ \
(r^{\vee}(p+h^{\vee}),u)=1,$$
where $r^{\vee}$ is the lacity of $\dot{\fg}$.
Note that $r^{\vee}(p+h^{\vee})$ is integral: $h^{\vee}$ is integral
for $\dot{\fg}\not=B(m|n), m\leq n$, and
$h^{\vee}=n-m+\frac{1}{2}$ for $\dot{\fg}=B(m|n), m\leq n$.

Let $k$ be a principal admissible level. Then
$\Delta_{re}(\ol{\fg})=\dot{\Delta}+\mathbb{Z}u{\delta}$,
where $u$ is as above  and the formula~(\ref{01}) takes the form
\begin{equation}\label{1}
Re^{\rho}\ch L_{\fg}(k\Lambda_0)=\ol{R} e^{\ol{\rho}}\ch L_{\ol{\fg}}(p\ol{\Lambda}_0).
\end{equation}
Note that $\Delta_{re}(\ol{\fg})\cap \Delta^+$ vhas the base $\dot{\Sigma}\cup\{\alpha_0'\}$, where
$$\alpha'_0=(u-1)\delta+\alpha_0,$$
where $\Sigma=\dot{\Sigma}\cup\{\alpha_0\}$.

Recall that for $x\in\dot{\fg}$ such that  $\dot{\Delta}_x$ is non-empty,
$\dot{\fg}$ and $\DS_x(\dot{\fg})$ have the same dual Coxeter numbers.
If, in addition, $\dot{\Delta}_x$ has rank more than one, then
 $\dot{\fg}$ and $\DS_x(\dot{\fg})$ have the same lacity $r^{\vee}$, so
 the principal admissible levels for $\fg$ and $\DS_x(\fg)$ coincide.
If $\dot{\Delta}_x$ has rank one,
 then the lacity of $\dot{\fg}$ is $1$ for $A(n\pm 1|n)$ and $2$ for other cases, 
whereas the lacity of $\DS_x(\dot{\fg})$
is $1$; hence  each principal admissible levels for $\fg$
is principle admissible for  $\DS_x(\fg)$.

\subsection{Vacuum modules for principal admissible levels}\label{veramdsup}
Retain notation of~\S~\ref{admsuper}.

\subsubsection{}\label{fibich}
Take $x\in\dot{\fg}$ satisfying~(\ref{eqx}) such that $x$ has a maximal rank, i.e.
 $$\dot{\ft}:=\DS_x(\dot{\fg})$$
has zero defect.  We denote by $I_{\ft}(k)$ the maximal proper submodule of $Vac^k_{\ft}$.
Let $k$ be an admissible level for $\ft$.
The vacuum module $Vac^k_{\ft}$ has a singular vector of
weight $r_0'.k\Lambda_0$, where
$$r_0':=r_{\alpha_0'}\in W.$$

From~\cite{F} it follows that in the case when $\dot{\ft}$ is a simple Lie algebra,
this singular vector generates $I_{\ft}(k)$.

\subsubsection{}
\begin{thm}{thmadmlevel}
Let $\dot{\fg}\not=B(n+1|n)$
be a finite-dimensional Kac-Moody algebra and
let $k$ be a principal admissible level.
Let $x\in\dot{\fg}_{\ol{1}}$ be of the maximal rank.

Assume that $\ft:=\fg_x$ satisfies the following: $\dot{\ft}$ is simple,

(A1) $I_{\ft}(k)$ is generated by a singular vector of
weight $r_0'.k\Lambda_0$;

(A2) any irreducible $V_k(\ft)$-module in the category $\CO$
is principal admissible.

Then

(i) $\DS_x(L(k\Lambda_0)\cong L_{\fg_x}(k\Lambda_0)$ as $\fg_x$-modules;

(ii) $\DS_x(V_k(\fg))\cong V_k(\fg_x)$ as vertex algebras;

(iii) for any $V_k(\fg)$-module $N$, $\DS_x(N)$ is  a
$V_k(\fg_x)$-module;

(iv) if $N$ is a $V_k(\fg)$-module in $\CO$, then $\DS_x(N)$ is either zero or
the direct sum of principal admissible modules of level $k$.
\end{thm}

\subsection{Proof of~\Thm{thmadmlevel}}
Note that $\dot{\fg}\not=A(m|n), B(m|n), D(m|n)$ with $m=n,n+1$ and
$D(n+2|n)$.
Using~\Rem{choicexnew}, we assume for that
$S,\dot{\Sigma}$ satisfies (P1), (P2), (P3) of~\S~\ref{app1}, i.e.

 $$S\subset\dot{\Sigma},\ \ (S,\theta)=0,\ \ ||\theta||^2=2,$$
where $\theta$
is the maximal root in $\Delta^+(\dot{\Sigma})$, and~(\ref{QPiS1}) holds.

 In particular, $\alpha_0:=\delta-\theta$ is the affine root for $\fg$ and for
  $\ft$.

 Since $L(k\Lambda_0)$ is $\dot{\fg}$-integrable,
 we can (and will) assume that $supp(x)=S$.

 We fix the $\mathbb{Z}_2$-grading on $Vac^k$ and all its subquotients by
letting the highest weight vector to be even. For a $\mathbb{Z}_2$-graded space $E$
we write $\dim E=(a|b)$ if $\dim E_{\ol{0}}=a, \dim E_{\ol{1}}=b$.
Retain notation of~\S~\ref{admsuper}.

\subsubsection{}\label{Ik}
Denote by $I(k)$ the maximal submodule of $Vac^{k}$ and by
$I_{\ol{\fg}}(p)$  the maximal submodule of the vacuum $\ol{\fg}$-module $Vac_{\ol{\fg}}^{p}$. One has
$$Re^{-k\Lambda_0}ch I(k)=Re^{-k\Lambda_0}(ch Vac^k-ch L(k\Lambda_0))=
\dot{R}-R e^{-k\Lambda_0}ch L(k\Lambda_0),$$
where $R, \dot{R},\ol{R}$ are the Weyl denominators for $\Delta^+, \dot{\Delta}^+, \ol{\Delta}^+$ respectively; recall that $\dot{\Delta}\subset\ol{\Delta}$,
so $\dot{\ol{R}}=\dot{R}$.

From~(\ref{1}) we have
$R e^{-k\Lambda_0}ch L(k\Lambda_0)=
\ol{R}e^{-p\ol{\Lambda}_0}ch L_{\ol{\fg}}(p\ol{\Lambda}_0)$.
This gives
\begin{equation}\label{II}
Re^{-k\Lambda_0}ch I(k)=\dot{R}-\ol{R}e^{-p\ol{\Lambda}_0} ch L_{\ol{\fg}}(p\ol{\Lambda}_0)
=\ol{R}e^{-p\ol{\Lambda}_0}ch I_{\ol{\fg}}(p).
\end{equation}

By~\S~\ref{vacmod}, $I_{\ol{\fg}}(p)$
is generated by a singular vector $v'$ of the weight
$$\mu:=p\ol{\Lambda}_0-(p+1)\alpha_0'.$$

Now the formula~(\ref{II}) can be rewritten as
\begin{equation}\label{III}
Re^{-r_0'.k\Lambda_0}ch I(k)=\ol{R}e^{-\mu}ch I_{\ol{\fg}}(p)
\end{equation}
since
$\ \
k\Lambda_0+\mu-p\ol{\Lambda}_0=k\Lambda_0-(p+1)\alpha_0'=r_0'.(k\Lambda_0)$.

Recall that $v'=f_0^{p+1}\vac$, where $f_0\in\fg_{-\alpha_0'}$.
For any $\beta\in S$ we have $(\alpha_0',\beta)=0$,
so $[\fg_{-\alpha_0'},\fg_{\pm\beta}]=0$. Therefore
  $\fg_{\pm\beta}v'=0$ and so
\begin{equation}\label{dimI}
\dim I_{\ol{\fg}}(p\Lambda_0)_{\mu}=(1|0),\ \
 I_{\ol{\fg}}(p\Lambda_0)_{\mu-\gamma}=0\ \text{ for }
\gamma\in \mathbb{Z}S\setminus\{0\}.\end{equation}

\subsubsection{}
Set
$$\cR:=\{\sum_{\nu\in\mathbb{Z}_{\geq 0}\Sigma} a_{\nu}e^{-\nu}|\ a_{\nu}\in\mathbb{C}\},\ \ \
P_S(\sum_{\nu\in\mathbb{Z}_{\geq 0}\Sigma} a_{\nu}e^{-\nu}):=\sum_{\nu\in\mathbb{Z}_{\geq 0}S} a_{\nu} e^{-\nu}.$$

Clearly, $\cR$ has a ring structure; this ring does not have zero divisors.
Note that $P_S$ is a ring homomorphism (since $S\subset \Sigma$) and $P_S^2=P_S$.

The ring $\cR$ contains $R,\dot{R},\ol{R},R^{-1},\ol{R}^{-1}$ and
$$P_S(R)=P_S(\dot{R})=P_S(\ol{R}).$$

Since $I_{\ol{\fg}}(p)$ is generated by a singular vector of weight $\mu$, one has
$e^{-\mu}ch I_{\ol{\fg}}(p)\in\cR$. By~(\ref{dimI}),
$P_S(e^{-\mu}ch_{\ol{\fg}} I(p))=1$. Using~(\ref{III}) we get

\begin{equation}\label{PSIk}
e^{-r_0'.k\Lambda_0} ch I(k)\in \cR,\ \ P_S(e^{-r_0'.k\Lambda_0}ch I(k))
=1,\end{equation}

By~(\ref{III}), $r_0'.k\Lambda_0$ is the highest weight of $I(k)$ and
$\dim I(k)_{r_0'.k\Lambda_0}=1$. Thus $I(k)_{r_0'.k\Lambda_0}$
  is spanned by an even singular vector.
By~(\ref{PSIk}), this vector has non-zero image in $\DS_x(I(k))$
$\DS_x(I(k))_{r_0'.k\Lambda_0}$
is spanned by this image.

We conclude that $\DS_x(I(k))_{r_0'.k\Lambda_0}$
is spanned by an even singular vector, which we denote by $v_0$.

\subsubsection{}\label{impsi}
Recall that  $\DS_x(Vac^k)=Vac_{\ft}^k$.
Consider the short exact sequence
$$0\to I(k)\to Vac^k\to Vac_k\to 0$$
and the corresponding long exact sequence
$$0\to E\to \DS_x(I(k))\xrightarrow{\phi}Vac_{\ft}^k \xrightarrow{\psi}  \DS_x(Vac_k)\to \Pi(E)\to 0.$$

By (A1), $I_{\ft}(k)$ is generated by a singular vector  $v_0'$
of weight $r_0'.k\Lambda_0$. Since $v_0,v_0'$ are singular,
$\phi(v_0)$ is proportional to $v_0'$.
There are two possibilities: either
 $\phi(v_0)=v_0'$ (up to a non-zero scalar) or $\phi(v_0)=0$.

Assume that $\phi(v_0)=0$. Since $v_0$ spans $\DS_x(I(k))_{r_0'.k\Lambda_0}$
one has $v_0'\not\in Im\phi=Ker\psi$. Since $v_0\in Ker\phi$ one has
\begin{equation}\label{EE}
\dim E_{r_0'.k\Lambda_0}=\dim \DS_x(I(k))_{r_0'.k\Lambda_0}=(1|0).
\end{equation}
Since $v'_0\not\in Ker\psi$, the $\ft$-module $\DS_x(Vac_k)$
has an even indecomposable
subquotient of length two with the socle $L_{\ft}(r_0'.k\Lambda_0)$ and the cosocle $L_{\ft}(k\Lambda_0)$. Since $Vac_k\cong L(k\Lambda_0)$ is self-dual,
$\DS_x(Vac_k)$ is also self-dual (see~\S~\ref{dual}); thus $\DS_x(Vac_k)$
 has an even indecomposable
subquotient of length two with the cosocle $L_{\ft}(r_0'.k\Lambda_0)$ and the socle $L_{\ft}(k\Lambda_0)$. Since $I_{\ft}(k)$ is generated by a singular vector
of weight $r_0'.k\Lambda_0$, one has
$[Vac^k_{\ft}: L_{\ft}(r_0'.k\Lambda_0)]=1$, so $ Im\psi$ does not have such subquotient.  Then
$\Pi(E)$ has an even subquotient  $L_{\ft}(r_0'.k\Lambda_0)$, which contradicts
to~(\ref{EE}).

We conclude that $\phi(v_0)=v_0'$ up to a non-zero scalar.
Denote by $a$
a preimage of $v_0$ in $I(k)\subset Vac^k$.
Let $N$ be a $V_k(\fg)$-module and  $\DS_x(N)\not=0$.
 Since $Vac_k=Vac^k/I(k)$, \S~\ref{vertmod} gives
$Y(a,z)N=0$, so $Y(v_0,z)\DS_x(N)=0$.
Since
$Vac_{\ft,k}=Vac^k_{\ft}/I_{\ft}(k)$
with $I_{\ft}(k)$ generated by $v_0'$,~\S~\ref{vertmod} implies that
$\DS_x(N)$ is a $V_k(\ft)$-module. This establishes (iii).

Let us prove that  $\DS_x(Vac_k)=L_{\ft}(k\Lambda_0)$.
Clearly, $[\DS_x(Vac_k):L_{\ft}(k\Lambda_0)]=1$.
Let $L_{\ft}(\lambda'')$
be a subquotient of $\DS_x(Vac_k)$ and $\lambda''\not=k\Lambda_0$.
By (iii) and (A2), $\lambda''$ is a $\ft$-admissible weight. One has
$\lambda''=k\Lambda_0-(\nu|_{\fh_S})$ for some
 $\nu\in(\mathbb{Z}_{\geq 0}\Sigma\cap S^{\perp})$. Recall that $(S,\dot{\Sigma})$
satisfies~(\ref{QPiS1}), so
$\nu|_{\fh_S}\in \mathbb{Z}\Sigma_S$, which contradicts to~\Lem{lemtt2}.
 This gives (i); (ii) follows from~\Thm{thmDSvert}
 (ii).
 \qed

\subsection{}
\begin{cor}{}
Let $\dot{\fg}$ is one of the following algebras: $A(m|n), C(n)$; $B(m|n), m\geq n+2$;
$D(m|n), m\not=n+1,n+2$, $B(n|n),F(4)$ or $G(2)$.
Take $x\in\dot{\fg}_{\ol{1}}$ such that $supp(x)$ is maximal.  Let $k$ be an
admissible level. Then

(i) $\DS_x(L(k\Lambda_0)\cong L_{\fg_x}(k\Lambda_0)$ as $\fg_x$-modules;

(ii) $\DS_x(V_k(\fg))\cong V_k(\fg_x)$ as vertex algebras;

(iii) for any $V_k(\fg)$-module $N$, $\DS_x(N)$ is  a
$V_k(\fg_x)$-module;

(iv) if $N$ is a $V_k(\fg)$-module in $\CO$, then $\DS_x(N)$ is either zero or
the direct sum of principal modules of level $k$.
\end{cor}
\begin{proof}
For $\dot{\fg}_x\not=0$, the assumption (A1)
of~\Thm{thmadmlevel} follows from~\cite{F}
and  the assumption (A2) follows from~\Thm{ThmAr}.
This gives (i)--(iii) for $\dot{\fg}_x\not=0$; (iv) follows from
(iii) and~\Thm{ThmAr}.

If $\dot{\fg}_x=0$, then (i)  is a particular case
 of~\Thm{thmABC} (i); moreover, (ii)-(iv) follow from (i).
\end{proof}

\subsection{}
The following lemma was used in the proof.

\begin{lem}{lemtt2}
Let $\dot{\ft}$ has zero defect and
 let $k$ be a principal admissible  level.
If $\lambda$
 is an admissible weight such that $k\Lambda_0-\lambda\in \mathbb{Z}\Sigma$,
 then $\lambda=k\Lambda_0$.
\end{lem}
 \begin{proof}
 Since $k\Lambda_0-\lambda\in \mathbb{Z}\Sigma$ one has
 $\Delta_{re}(\lambda)=\Delta_{re}(k\Lambda_0)$.
 Set
 $$\lambda':=\lambda+(p+h^{\vee})(1-\frac{1}{u})\Lambda_0.$$
 One readily sees that $\lambda'$ is a dominant weight of level $p$.
 One has $p\Lambda_0-\lambda'=k\Lambda_0-\lambda\in \mathbb{Z}\Sigma$. Since $\lambda'$ is dominant,
$$0\leq (\Lambda_0,p\Lambda_0-\lambda')=-(\lambda',\Lambda_0)\ \text{ and }\
0\leq (p\Lambda_0-\lambda',\lambda')=p(\Lambda_0,\lambda')
-(\lambda',\lambda').$$
Therefore $(\lambda',\lambda')=(\Lambda_0,\lambda')=0$.
Since $p\Lambda_0-\lambda'=k\Lambda_0-\lambda\in \mathbb{Z}\Sigma$, we obtain
$k\Lambda_0-\lambda=0$ as required.
 \end{proof}

\section{Appendix}
\label{app1}
We fix the standard triangular decomposition in $\dot{\Delta}_{\ol{0}}$ and consider the bases $\dot{\Sigma}$ which are compatible
with this decomposition. For each base $\dot{\Sigma}$ denote by $\theta_{\dot{\Sigma}}$
the maximal root of $\Delta^+(\dot{\Sigma})$.

Let $\cS$ be the set
of maximal isotropic subsets of $\Delta_{\ol{1}}$.
Consider the action of the Weyl group $\dot{W}$ on $\cS$.
For each orbit  it is not hard to give
an example of a pair $(S,\dot{\Sigma})$ such that $S$ is a representative of this orbit and

(P1) $S\subset\dot{\Sigma}$;

(P2) $\theta_{\dot{\Sigma}}\in\dot{\Delta}^{\#}$ and $(\theta_{\dot{\Sigma}},S)=0$ for $\dot{\fg}\not=A(m|n), B(m|n), D(m|n)$ with $m=n,n+1$;

(P3) if $\dot{\fg}\not=D(n+1,n), D(n+2|n)$, then the following inclusion holds
\begin{equation}\label{QPiS1}
 (\mathbb{Q}_{\geq 0}\Sigma\cap S^{\perp})\subset (\mathbb{Q}S+\mathbb{Q}_{\geq 0}\Sigma_S),
 \end{equation}
 where
$\Sigma=\{\delta-\theta_{\dot{\Sigma}}\}\cup\dot{\Sigma}$ is a base for
$\Delta=\dot{\Delta}^{(1)}$.

For instance, for $B(m,n), D(m,n), n>m$ one has $\dot{\Delta}^{\#}=C_n$.
We take $S:=\{\vareps_{i}-\delta_{i+1}\}_{i=1}^m$ and
$$\dot{\Sigma}:=\{\delta_1-\vareps_1,\vareps_1-\delta_2,\ldots,
\vareps_m-\delta_{m+1},\delta_{m+1}-\delta_{m+2},\ldots,\delta_{n-1}-\delta_n,a\delta_n\},$$
where $a=1$ for $B(m|n)$ and $a=2$ for $D(m|n)$. One has $\theta=2\delta_1$, so (P1), (P2) are satisfied. One has
$$\dot{\Sigma}_S=\{\delta_1-\delta_{m+2},\ldots,\delta_{n-1}-\delta_n,a\delta_n\},$$
of type $B(0|n-m)$ for $B(m|n)$ and $C_{n-m}$ for $D(m|n)$; it is easy to see that~(\ref{QPiS1}) holds.

\end{document}